\documentclass[a4paper]{amsart}
\usepackage{amssymb}
\usepackage[hyphens]{url} 
\usepackage{color}
\usepackage[english]{babel}
\newcommand{\m}{\mathcal{M}}
\newcommand{\mo}{\mathcal{M}_0}
\newcommand{\mt}{\mathcal{M}_t}
\newcommand{\tb}{\bar{t}}
\newcommand{\xb}{\bar{x}}
\newcommand{\rb}{r_B}
\newcommand{\wb}{\overline{W}}
\newcommand{\wbi}{\underline{W}}
\newcommand{\pb}{\bar{p}}
\newcommand{\qb}{\bar{q}}

\newtheorem{theorem}{Theorem}[section]
\newtheorem{lem}[theorem]{Lemma}
\newtheorem{prop}[theorem]{Proposition}
\newtheorem{cor}[theorem]{Corollary}

\theoremstyle{definition}

\theoremstyle{remark}

\numberwithin{equation}{section}

\begin{document}

\title[Volume preserving non homogeneous mean curvature flow]{Volume preserving non homogeneous mean curvature flow of convex hypersurfaces}
\author{Maria Chiara Bertini \and Carlo Sinestrari}

\date{}
\begin{abstract}
We consider a convex Euclidean hypersurface that evolves by a volume or area preserving flow with speed given by a general nonhomogeneous function of the mean curvature. For a broad class of possible speed functions, we show that any closed convex hypersurface converges to a round sphere. The proof is based on the monotonicity of the isoperimetric ratio, which allows to control the inner and outer radius of the hypersurface and to deduce uniform bounds on the curvature by maximum principle arguments.
\end{abstract}
\maketitle
\noindent {\bf MSC 2010 subject classification} 53C44, 35B40 \bigskip
\section{Introduction}
Let $\m$ be an oriented, compact  $n$-dimensional  manifold without boundary. We embed $\m$ in the Euclidean $(n+1)$-space by $F_0:\m\rightarrow \mathbb{R}^{n+1}$, and denote its image by $\mo=F_0(\m)$. We assume that $\mo$ is strictly convex.
Then we consider a family of maps $F:\m\times[0,T)\rightarrow\mathbb{R}^{n+1}$, with $F_t:=F(\cdot,t):\m\rightarrow \mathbb{R}^{n+1}$  satisfying \smallskip
\begin{equation}\label{fl}
\left\{
\begin{array}{l}
\partial_t F(x,t)=[-\phi(H(x,t))+h(t)]\nu(x,t) \medskip\\
F(x,0)=F_0(x),\\
\end{array}
\right.
\end{equation}
where:
\begin{itemize}
\item $H$ and $\nu$ denote respectively the mean curvature and the outer unit normal  vector of the evolving hypersurface $\mt:=F_t(\m)$.
 \item $\phi:[0,+\infty)\rightarrow \mathbb{R}$ is a continuous function, $C^2$ differentiable in $(0,+\infty)$ with the following properties: \medskip
 \begin{trivlist}
        \item $i)$ $\phi(0)=0$, \hspace{0.5cm} $\displaystyle \lim_{\alpha\to \infty}\phi(\alpha)=\infty$; \medskip
        \item $ii)$ $\phi'(\alpha)>0 \hspace{5mm}\forall \alpha>0$; \medskip
        \item $iii)$ $\displaystyle \lim_{\alpha \to 0}\frac{\phi'(\alpha)\alpha^2}{\phi(\alpha)}=0$, \hspace{0.5cm} $ \displaystyle\lim_{\alpha \to \infty}\frac{\phi'(\alpha)\alpha^2}{\phi(\alpha)}=\infty$; \medskip
        \item $iv)$  $ \displaystyle\lim_{\alpha\to 0}\phi'(\alpha)\alpha=0$; \medskip
        \item $v)$ $\phi''(\alpha)\alpha\geq-2\phi'(\alpha)\hspace{5mm}\forall \alpha>0$. \medskip
        \end{trivlist}
 \item The function $h(t)$ is either defined as
 \begin{equation}\label{vpr}h(t):=\frac{1}{|\mt|}\int_{\mt} \phi(H)d\mu\end{equation}
 or as
 \begin{equation}\label{apr}h(t):=\frac{\int_{\mt} H\phi(H)d\mu}{\int_{\mt} Hd\mu}.\end{equation}
\end{itemize}
The choice of $h$  is made in order to keep the volume enclosed by $\mt$ constant in case \eqref{vpr}, and in order to keep the area of $\mt$ constant in case \eqref{apr}. Flows of this form are sometimes called {\em constrained} curvature flows, while the corresponding ones without the $h(t)$ term will be referred to as {\em standard} flows. 

We will prove the following result.
\begin{theorem}\label{mt}Let  $F_0:\m\rightarrow \mathbb{R}^{n+1}$, with $n \geq 1$, be a smooth embedding of an oriented, compact   $n$-dimensional  manifold without boundary, such that $F_0(\m)$ is strictly convex. Then the flow \eqref{fl} with $h(t)$ given by \eqref{vpr} (resp. \eqref{apr}) has a unique smooth solution, which exists for any time $t\in[0,\infty)$. The solution is convex and converges smoothly, as $t\to\infty$, to a round sphere that encloses the same volume (resp. has the same area) as the initial datum $\mo$.
\end{theorem}

This theorem can be regarded as a generalization of the result in \cite{Si}, where the case $\phi(\alpha)=\alpha^k$ with $k>0$ was considered. Here we are able to treat a more general class of speeds depending on the mean curvature, where no assumption of homogeneity or convexity/concavity is made. The main assumption is the positivity of $\phi'$, which ensures the parabolicity of the problem. The additional requirements we put on $\phi$ are satisfied in most of the natural examples. For example, linear combinations of powers
$$\phi(\alpha)=\sum_{i=1}^{l}c_i\alpha^{k_i}\hspace{3mm}c_i, k_i >0$$
satisfy assumptions $i)$-$v)$. There are also easy examples with non polynomial growth: for example
$$\phi(\alpha)=\log(1+\alpha) \hspace{3mm}\text{or}\hspace{3mm} \phi(\alpha)=e^{\alpha}-1$$
 satisfy our hypotheses.

The behaviour of convex hypersurfaces evolving by geometric flows has been widely studied in the last decades, starting from the paper by Huisken \cite{Hu1} on the mean curvature flow in the standard case, where it was proved that closed convex hypersurfaces contract to a point in finite time and become spherical after rescaling. Shortly afterwards, \cite{Hu2}, Huisken proved the corresponding result in the volume-preserving case, where the convergence to a sphere takes place in infinite time and without rescaling. Since then, many authors have obtained convergence results for convex hypersurfaces under various geometric flows, both in the standard and in the constrained case, see e.g. \cite{CaSi} or \cite{Si} for references. While in the literature on geometric flows the velocity is usually assumed to be a homogeneous function of the principal curvatures, nonhomogeneous speeds of the form $\phi(H)$ have been sometimes considered in the past. We recall in particular the paper by Smoczyk \cite{Sm} where the validity of differential Harnack inequalities was studied, and the one by Alessandroni and Sinestrari \cite{AlSi} where the singular profile of mean convex solutions was investigated for a particular class of $\phi$. On the other hand, some convergence results of convex hypersurfaces for standard flows driven by nonhomogeneous speeds have been obtained in the expanding case by Chow and Tsai \cite{CT1,CT2}.

Compared with most of the previous works on constrained flows, this paper follows a different approach. Usually, in fact, see e.g. \cite{CaSi, Hu2, Mc1, Mc2}, the convergence to a spherical profile is obtained by extending, with some additional effort, the techniques of the standard case, which typically rely on the invariance or the improvement of the pinching of the principal curvatures of the hypersurface. Our proof instead, like the ones of \cite{An1b,Si}, does not employ any pinching condition and exploits the monotonicity of the isoperimetric ratio of the hypersurface under the flow, which is a peculiar property of the volume/area-preserving case. In this respect, constrained flows exhibit a better behaviour than the standard ones. We underline that, without the $h(t)$ term, general convergence results for convex hypersurfaces are only known for speeds which are $1$-homogeneous functions of the curvature, or under some dimension restrictions, see the references in \cite{Si}.

The paper is organized as follows. After recalling some notation in Section \ref{Preliminaries}, in Section \ref{Boundedness} we show the monotonicity of the isoperimetric ratio and the preservation of convexity under the flow. From this, we deduce a uniform control on the ratio between the outer and inner radius which allows us to bound the curvature from above by the maximum principle technique introduced by Tso \cite{T} and to prove long time existence. In Section \ref{Convergence}, we study the long time behaviour of the solution. Since we are not assuming any regularity of $\phi'(\alpha)$ up to $\alpha=0$, a difficulty in this step is the possible loss of uniform convexity as time goes to infinity. To handle this problem, some previous papers \cite{Sch2, CaSi, Si} employed some advanced regularity results on the equations of porous medium type, a procedure which would be difficult to generalize to the present context. We find instead another argument, which employs a variant of Tso's technique to prove directly that the curvature remains bounded away from zero. The applicability of the argument relies crucially on the presence of the $h(t)$ term, again showing a better behaviour of the constrained case.

\section{Preliminaries}\label{Preliminaries}
\subsection*{Notation}
Let $F:\m\rightarrow \mathbb{R}^{n+1}$ be an embedded hypersurface with local coordinates $(x^1,\cdots, x^n)$. We endow $\m$ with the induced metric $g=(g_{ij})$ given by
$$g_{ij}=\left(\frac{\partial F}{\partial x^i},\frac{\partial F}{\partial x^j}\right)$$ where $(\cdot,\cdot)$ is the standard Euclidean inner product. We also denote respectively by $\nabla$ and $A=(h_{ij})$ the Levi-Civita connection and the second fundamental form of $\m$, while the measure is $d\mu=\sqrt{\det g_{ij}}\, dx$. The principal curvatures are denoted by $\lambda_1,\dots,\lambda_n$, and the mean curvature by $H=\lambda_1+\dots+\lambda_n$. We say that the hypersurface is strictly convex if all $\lambda_i$'s are positive. We denote by  $\Delta=g^{ij}\nabla_i\nabla_j$ the Laplace-Beltrami operator, where $g^{-1}=(g^{ij})$ is the inverse of the metric. As ususal, we always sum on repeated indices, and we lower or lift
tensor indices via $g$, e.g. the Weingarten operator is given by $$h^i_j=h_{kj}g^{ik}.$$
Given tensors $T=(T^{i_1\dots i_s}_{j_1\dots j_r})$ and $S=(S^{i_1\dots i_s}_{j_1\dots j_r})$ on $\m$, we use brackets to denote their inner product
$$\langle T,S\rangle=T^{i_1\dots i_s}_{j_1\dots j_r}S_{i_1\dots i_s}^{j_1\dots j_r}.$$
In particular, the square of the norm is given by
$$|T|^2=T^{i_1\dots i_s}_{j_1\dots j_r}T_{i_1\dots i_s}^{j_1\dots j_r}.$$
%In the following we will use the same notation for the norm with respect to $g$ and the absolute value of a scalar quantity.
Given a point $q\in \mathbb{R}^{n+1}$, the {\em support function} of $\m$ with respect to $\xb$ is
$$u_{q}(x):=(F(x)-q,\nu(x)),$$
where $\nu(x)$ is the outer unit normal vector of $\m$ at the point $x$. The subscript $\xb$ will be omitted whenever there will be no ambiguity. 

\subsection*{Short time existence and evolution equations}
It is well known that a flow of the form \eqref{fl} without the volume preserving term is parabolic if at any point
\begin{equation}\label{par}
\frac{\partial\phi}{\partial\lambda_i}>0,\hspace{1cm}i=1,\dots,n
\end{equation}
i.e. $\phi'>0$ at any point. This is guaranteed, by condition $ii)$ on $\phi$, if $H>0$ at any point.
Parabolicity ensures the local existence and uniqueness of the solution.
The additional term $h(t)$ only depends on time and does not interfere with the parabolicity of the equation. Hence, we have the following result, see
 \cite{Hu2, HuPo,Mc2} for more details.
\begin{theorem} Let $F_0:\m\rightarrow \mathbb{R}^{n+1}$ be a smooth embedding of an oriented, compact  $n$-dimensional  manifold without boundary, such that $F_0(\m)$ is strictly mean convex. Then the flow \eqref{fl} has a unique smooth solution $\mt$ defined on a maximal time interval $[0,T)$.
\end{theorem}

In the next proposition we list the evolution equations for the main geometrical quantities associated with the flow \eqref{fl}, which can be computed similarly to \cite{Hu1}, see also \cite{AlSi,Sm}. For brevity, the dependence of the functions by space and time is omitted. We remind in particular that $\phi$ is space and time dependent, while $h$ is only time dependent.

\begin{prop}
We have the following evolution equations for the flow \eqref{fl}:
\begin{align*}
&\partial_t g_{ij} = 2(-\phi +h)h_{ij},\\
&\partial_t g^{ij} = -2(-\phi +h)h^{ij},\\
&\partial_t \nu=\nabla\phi,\\
&\partial_t d\mu = H(-\phi+h)d\mu,\\
&\partial_t h^i _j = \phi'\Delta h^i _j+\phi'' \nabla^i H\nabla_j H+\phi'|A|^2h^i _j+(\phi-h-H\phi')h^i _kh^k _j,\\
&\partial_t H = \phi'\Delta H +\phi''|\nabla H|^2+(\phi-h)|A|^2,\\
&\partial_t\phi = \phi'\Delta\phi+\phi'(\phi-h)|A|^2,\\
&\partial_t u = \phi'\Delta u+\phi'|A|^2u-\phi-\phi'H+h.
\end{align*}
\end{prop}

\section{Boundedness of the velocity and long time existence}\label{Boundedness}

\subsection*{Some geometrical bounds}
An important feature of the volume/area-preserving flows we are considering is that the isoperimetric ratio of the hypersurface is non increasing in time, a property which was observed by M. Gage \cite{G} for the area preserving mean curvature flow. We recall that if $\Omega\subset\mathbb{R}^{n+1}$ is a compact, convex set with nonempty interior, the isoperimetric ratio is given by
$$\mathcal{I}(\Omega)=\frac{|\partial\Omega|^{n+1}}{|\Omega|^n},$$
where $|\Omega|$ is the $(n+1)$-dimensional measure of $\Omega$ and $|\partial\Omega|$ is the $n$-dimensional measure of its boundary $\partial\Omega$. Powers are chosen in order to make $\mathcal{I}(\Omega)$ invariant by homotheties.
 It is well known that the isoperimetric ratio satisfies the isoperimetric inequality
\begin{equation}\label{dis}\mathcal{I}(\Omega)\geq(n+1)^n\omega_n,
\end{equation}
where $\omega_n=|\mathbb{S}^n|$. Denote by $\Omega_t$ the region enclosed by $\mt$.

\begin{lem}\label{l1} For the flow \eqref{fl} we have
$$\frac{d}{dt}|\mt|\leq 0\hspace{2cm}\text{in case of $h$ given by \eqref{vpr},}$$
$$\frac{d}{dt}|\Omega_t|\geq0\hspace{2cm}\text{in case of $h$ given by \eqref{apr}.}$$
\end{lem}
\begin{proof}
We start from the volume preserving case. For any $t$, let us denote by $\bar{H}=\bar{H}(t)$ the value such that
$\phi(\bar{H})=\frac{1}{|\mt|}\int_{\mt} \phi$, which is uniquely defined by the monotonicity of $\phi$. Then $\int_{\mt}[\phi(\bar{H})-\phi(H)]=0$, and so \begin{align*}
\frac{d}{dt}|\mt| &=  \int_{\mt} H(-\phi(H)+h) \, d\mu =
\int_{\mt}[H\phi(\bar{H})-H\phi(H)] \, d\mu\\
&=\int_{\mt}[H-\bar{H}][\phi(\bar{H})-\phi(H)] \, d\mu \\
&= \int_{H\geq\bar{H}}[H-\bar{H}][\phi(\bar{H})-\phi(H)] \, d\mu+\int_{H\leq\bar{H}}[H-\bar{H}][\phi(\bar{H})-\phi(H)] \, d\mu.
\end{align*}
Since both terms on the right side are nonpositive, the assertion follows.\\
Analogously, for the area preserving flow, let $\bar{H}=\frac{1}{|\mt|}\int_{\mt}H$. We have
\begin{align*}
\frac{d}{dt}|\Omega_t| &=  \int_{\mt}[-\phi(H) + h]  \, d\mu \\
&= \frac{|\mt|}{\int_{\mt}H \, d\mu}\left(-\int_{\mt}\phi(H)\bar{H} \, d\mu+\int_{\mt}\phi(H)H \, d\mu \right) \\
&= \frac{|\mt|}{\int_{\mt}H \, d\mu}\int_{\mt}[\phi(H)-\phi(\bar{H})][H-\bar{H}] \, d\mu,
\end{align*}
which is nonnegative by an argument similar to the previous case.
\end{proof}
From Lemma \ref{l1} and the isoperimetric inequality \eqref{dis}, we deduce the following.
\begin{cor}\label{al}For the flow \eqref{fl} with $h$ given either by \eqref{vpr} or \eqref{apr} there exist constants $M_1, M_2, V_1,V_2 >0$ depending only on $\mo, \Omega_0$ and $n$ such that
$$M_1\leq |\mt|\leq M_2,\hspace{1.5cm}V_1\leq |\Omega_t|\leq V_2.$$
\end{cor}
%\begin{proof}
%From Lemma \ref{l1} and \eqref{dis} it follows that, in case of $h$ given by \eqref{vpr},
%$$|\mo|\geq|\mt|=\mathcal{I}(\Omega_t)^{\frac{1}{n+1}}|\Omega_0|^{\frac{n}{n+1}}\geq[(n+1)^n\omega_n|\Omega_0|^n]^{\frac{1}{n+1}}.$$
%The area preserving case is analogous.
%\end{proof}

The bound on $\mathcal{I}(\Omega_t)$ given by Lemma  \ref{l1} also allows us to control the shape of the hypersurface in terms of the inner ad outer radii. We recall that if $\Omega\subset\mathbb{R}^{n+1}$ is a compact, convex set with nonempty interior, its inner [resp. outer] radius  is the radius of the biggest $(n+1)$-dimensional sphere contained in $\Omega$ [resp. the smallest $(n+1)$-dimensional sphere that contains $\Omega$]. We indicate inner and outer radii respectively by  $R_-(\Omega)$ and $R^+(\Omega)$.
 We need the following property (see \cite[Proposition 5.1]{An1b} or \cite[Lemma 4.4]{HuSi}).
\begin{prop}\label{p1}
For any $n\geq1$ and $c_1>0$ there exist $c_2=c(c_1,n)$ with the following property. Let $\Omega\subset\mathbb{R}^n$ be a compact, convex set with non empty interior such that $\mathcal{I}(\Omega)\leq c_1$. Then $\Omega$ satisfies
$$\frac{R^+(\Omega)}{R_-(\Omega)}\leq c_2.$$
\end{prop}

\subsection*{Preservation of convexity}
Now we show that, if the initial datum $\mo$ is strictly convex, then strict convexity is preserved for all time such that the flow is defined.
\begin{prop}\label{conv}If $\mo$ is strictly convex, then $\mt$ is strictly convex for all $t\in[0,T)$.
\end{prop}
\begin{proof}
It is sufficient to show that, if $\mt$ is strictly convex on $[0,T^*)$, for an arbitrary $T^*<T$, then $\mathcal{M}_{T^*}$ is strictly convex. On the interval $[0,T^*)$, let $b^i_j:=(h^i_j)^{-1}$ the inverse of the Weingarten operator. By standard computations, we get the evolution equation of $b^i_j$:
\begin{align*}
\partial_t b^i_j&=\phi'\Delta b^i_j-2\phi'h^m_n\nabla_{l}b^n_j\nabla^{l}b^i_m-\phi''(b^i_m\nabla^m H)(b^n_j\nabla_n H)\\
&\hspace{0.5cm}-\phi'|A|^2b^i_j+(\phi' H-\phi +h)\delta^i_j.
\end{align*}
In order to prove that gradient terms give a negative contribution, we
rewrite the gradients of $b^i_j$ in terms of gradients of $h^i_j$:
\begin{equation*}\label{gb}
h^m_n\nabla_lb^n_j\nabla^lb^i_m=b^q_jb^i_pb^r_m\nabla_lh^m_q\nabla^lh^p_r.
\end{equation*}
Then we use the following inequality proved by Schulze (see the second-last formula in the proof of Lemma 2.5 of \cite{Sh}
with $k=1$)
$$-Hb^r_m\nabla_lh^m_q\nabla^lh^p_r\leq-\nabla^pH\nabla_qH,$$ 
which gives
\begin{equation}\label{cv}
\begin{split}
-&\phi'2h^m_n\nabla_{l}b^n_j \nabla^{l}b^i_m-\phi''(b^i_m\nabla^m H)(b^n_j\nabla_n H)\\
&\hspace{5mm}\leq-\frac{1}{H}(2\phi'+\phi''H)(b^i_m\nabla^mH)(b^n_j\nabla_nH)\leq0,
\end{split}
\end{equation}
where for the last inequality we used property $v)$ of $\phi$.
Furthermore, since $[0,T^*)$ is strictly contained in the existence time interval of the solution, there exists $H^*>0$ such that $0<H<H^*$ on $[0,T^*)$. Such a bound on $H$ also implies a bound from above on $\phi' H + h$, thanks also to property iv) of $\phi$. Then, using \eqref{cv} we obtain
$$\partial_t b^i_j\leq\phi'\Delta b^i_j-\phi'|A|^2 b^i_j+c_0,$$
where $c_0$ is a constant only depending on $n,\mo$ and $H^*$.

So, using the maximum principle, $b^i_j$ is bounded on the finite interval $[0,T^*]$, and then on such interval all principal curvatures stay bounded from below by a positive constant. Then, $\mathcal{M}_{T^*}$ is strictly convex and the assertion follows.
\end{proof}

Once we know that our solution remains convex, we can apply the results of the previous subsection. Let  us set $R_-(t):= R_-(\Omega_t)$ and $R^+(t):=R^+(\Omega_t)$.
\begin{cor}\label{c1} For a convex $\mt$ evolving by \eqref{fl}, there are positive constants $R_-$ and $R^+$ such that
$$R_-<R_-(t)\leq R^+(t)<R^+,$$
where $R_-$ and $R^+$ depend only on $n$, $|\mo|$ and $|\Omega_0|$.
\end{cor}
\begin{proof}
  By virtue of the boundedness of the isoperimetric ratio, we can use Proposition \ref{p1} to say that $\frac{R^+(t)}{R_-(t)}$ is uniformly bounded by a constant $c_2$ depending only on $n$, $|\mo|$ and $|\Omega_0|$. Then, comparing $|\Omega_t|$ with the volume of a ball and using Corollary \ref{al}, we find
  $$V_1\leq|\Omega_t|\leq\omega_n\frac{(R^+(t))^{n+1}}{n+1}\leq\omega_n\frac{(c_2R_-(t))^{n+1}}{n+1}
  \leq c_2^{n+1}|\Omega_t|\leq c_2^{n+1}V_2.$$
  Then we obtain bounds from both sides on $R_-(t)$ and $R^+(t)$.
\end{proof}

\subsection*{Upper bound on the velocity}
Thanks to Corollary \ref{c1} and Proposition \ref{conv}, we are now able to control uniformly the velocity of the flow, and obtain curvature bounds which imply the long time existence for the solution. To do this, we follow a method first introduced by Tso \cite{T} and adapted by Andrews and by McCoy \cite{An1b,Mc1} to the volume preserving setting.
\begin{lem}\label{lpi}Given $\bar{t}\in[0,T)$, let $\bar{q}\in\Omega_{\bar{t}}$ be such that $B(\bar{q}, R_-)\subset\Omega_{\bar{t}}$, where $R_-$ is taken as in Corollary \ref{c1}. Then
$$B(\bar{q},R_-/2)\subset\Omega_t\hspace{2cm}\forall t\in[\bar{t},\min\{\tb+\tau,T\})$$
for some constant $\tau>0$  that only depends on $n,|\mo|$ and $|\Omega_0|$.
\end{lem}
\begin{proof}
Define $r(x,t):=|F(x,t)-\qb|$ and set $u(x,t):=(F(x,t)-\qb,\nu(x,t))$. Then
\begin{equation}
\label{confr}\partial_t r = \frac{1}{2r} \partial_t r^2 =
 (h-\phi(H))\frac{u}{r} > -\phi(H)\frac{u}{r} > -\phi(H).
\end{equation}
Let $r_B(t)$ be the radius of the ball centered in $\qb$ and contracting by
\begin{align}\label{flr}r_B'(t)=-\phi \left(\frac{n}{r_B(t)} \right)\end{align}
 with initial datum $r_B(\tb)=R_-$.
Define $f(x,t):=r(x,t)-\rb(t)$. Using \eqref{confr}, we obtain
$$\partial_t f>-\phi(H)+\phi \left(\frac{n}{\rb} \right).$$
At time $\tb$, $f(\cdot,\tb)>0$. Suppose that there exists a first time $t^*>\bar{t}$ such that $f(x^*,t^*)=0$ at some point $x^*$. Then $\partial_t f(x^*,t^*) \leq 0$. In addition, the ball with radius $\rb(t^*)$ touches $\mt$ from the inside at the point $F(x^*,t^*)$, which implies
%Suppose that exists $(x^*,t^*)$, with $t^*>\bar{t}$ such that $f(x^*,t^*)=0$. Time $t^*$ is the first time such that the function $f$ vanishes, i.e. the ball with radius $\rb(t)$ touches $\mt$ from the interior. Then
$$H(x^*,t^*)\leq\frac{n}{\rb(t^*)} \ \Longrightarrow \ \phi(H(x^*,t^*))\leq\phi\left(\frac{n}{\rb(t^*)}\right).$$
The contradiction shows that,
%\hspace{2cm}\Delta r^2=-2H(F,v)+2n$$
%$$\Rightarrow(\partial_t-\frac{\phi}{H}\Delta)r^2=2h(F,n)-2n\frac{\phi}{H}$$
%And so
%\begin{align*}
%(\partial_t-\frac{\phi}{H}\Delta)r=&\frac{1}{2r}(\partial_t-\frac{\phi}{H}\Delta)r^2+
%\frac{\phi}{rH}|\nabla r|^2\geq\\
%\geq&\frac{1}{r}h(F,\nu)-n\frac{\phi}{rH}\geq-n\frac{\phi}{rH}
%\end{align*}
%Define $t_1$ as the first time such that $\xb$ exits from interior of the evolving hypersurfaces, i.e. $t_1:=\inf\{t>\tb | \xb\notin \Omega_t\}$. If such a time doesn't exist, Lemma is proved. Otherwise $t_1<T$,  so $\xb\in\mathcal{M}_{t_1}$.
%Define
%$$\rb(t)\min_\mt r(\cdot,t)\hspace$$
for every time $t$ where the flow \eqref{flr} is defined, we have $r(x,t)\geq\rb(t)$. It now suffices to choose $\tau>0$ such that $\rb(t)\geq\frac{R_-}{2}$ for every $t\in[\bar{t},\min\{T,\tb+\tau\})$. Notice that $\tau$ depends neither on the initial time $\tb$ nor on $\qb$.
\end{proof}

\begin{prop}\label{p2} At any time $t\in[0,T)$, we have
$$\phi(H)\leq C_1$$
where $C_1$ is a positive constant only depending on $n$ and $\mo$.
%$$\phi(H)\leq\max\{\max_{\mo}\phi(H),\frac{\phi(C)}{c}d\}$$
%where $c,C$ are positive constants  only depending on  $n$ and $\mo$, and $d:=\sup_{[0,T)}(\mathrm{diam}\mt)$
\end{prop}
\begin{proof}
By Lemma \ref{lpi}, for every $\tb\in[0,T)$, exists $\qb$ such that $$B(\bar{q},R_-/2)\subset\Omega_t\hspace{1cm}\forall t\in[\bar{t},\min\{T,\tb+\tau\}).$$
Let us set
$$u(x,t):=(F(x,t)-\qb,\nu(x,t)).$$
Choosing $c:=\frac{R_-}{4}$ we obtain, by the convexity of $\mt$,
\begin{align}
\label{du}c\leq u-c\leq d, \qquad \forall \, t \in [\bar{t},\min\{T,\tb+\tau\}),
\end{align}
where  $d=\sup_{[0,T)}(\mathrm{diam}\mt)$ is finite by Corollary \ref{c1}. Then, the function
$$W(x,t):=\frac{\phi(H(x,t))}{u(x,t)-c}$$
is well defined on $[\bar{t},\min\{T,\tb+\tau\})$. Standard computations show that
%\begin{eqnarray*}
%\partial_t W&=&\frac{(u-c)\partial_t\phi-\phi\partial_t u}{(u-c)^2}\\
%&=&\frac{(u-c)\phi'\Delta\phi-\phi\phi'\Delta(u-c)}{(u-c)^2}\hspace{2mm}+\\
%&&-\frac{\phi'}{u-c}h|A|^2-\frac{\phi}{(u-c)^2}\{h-(\phi'H+\phi)+c|A|^2\phi'\}
%\end{eqnarray*}

\begin{align*}
\partial_t W &=\frac{(u-c)\partial_t\phi-\phi\partial_t u}{(u-c)^2}\\
&= \frac{(u-c)\phi'\Delta\phi-\phi\phi'\Delta(u-c)}{(u-c)^2}\hspace{2mm}\\
&\hspace{0.5cm}-\frac{\phi'}{u-c}h|A|^2-\frac{\phi}{(u-c)^2}\{h-(\phi'H+\phi)+c|A|^2\phi'\}
\end{align*}
and
$$\phi'\Delta W=\frac{(u-c)\phi'\Delta\phi-\phi\phi'\Delta(u-c)}{(u-c)^2}
-\frac{2\phi'}{u-c}\langle\nabla W,\nabla u\rangle.$$
Now, define
$$\wb(t):=\sup_{\mt}W(x,t)\hspace{1.5cm}X(t):=\{x\in\m | W(x,t)=\wb(t)\}$$
where by ``$\sup_{\mt}$'' we mean the supremum taken on $\m\times\{t\}$.
Then, discarding the negative $h$ terms, we find that the upper Dini derivative $D_+\wb$ satisfies
\begin{align*}
D_+\wb&\leq\sup_{X(t)}\partial_t W \leq \sup_{X(t)}(\partial_t-\phi'\Delta)W\\
&\leq\wb^2+\wb\sup_{X(t)}\phi'\Bigl(\frac{H}{u-c}-\frac{c|A|^2}{u-c}\Bigl)\\
&\leq\wb^2+\wb\sup_{X(t)}\frac{\phi'H}{u-c}\Bigl(1-\frac{c^2 H}{nd}\Bigl)
\end{align*}
where for the last inequality we used convexity of $\mt$ and \eqref{du}.

%Suppose $H\geq C$ at some $(x^*,t^*)$, with
Let us choose $C$ large enough to satisfy
\begin{equation}\label{condh}
   \begin{cases}
   C\geq\frac{3nd}{c^2}\\
   \frac{c}{\phi(C)}<\tau\\
   \end{cases}
\end{equation}
so that $H \geq C$ implies that $1-\frac{c^2H}{nd}\leq-\frac{2c^2}{3nd}H$. Now, suppose that $\wb(t^*) \geq \phi(C)/c$ for some time $t^*$. Then, using the bound $u-c \geq c$ and the monotonicity of $\phi$ we have that $H(x^*,t^*) \geq C$ for any $x^* \in X(t^*)$.
Then, we get at time $t=t^*$
\begin{align*}
D_+\wb \leq\wb^2-\frac{2c^2}{3nd}\wb\sup_{X(t^*)}\frac{\phi' H^2}{u-c} =\wb^2\sup_{X(t^*)}\left\{1-\frac{2c^2\phi'H^2}{3nd\phi}\right\}.
\end{align*}
Also, by property $iii)$ of $\phi$, we can choose $C$ sufficiently big such that $H \geq C$ implies $$1-\frac{2c^2\phi'H^2}{3nd\phi}<-1.$$
Then 
$$D_+\wb\leq-\wb^2,$$
and so a standard comparison argument implies
\begin{equation}\label{w1}
W\leq\max\left\{\max_{\mo}W,\frac{\phi(C)}{c}\right\}\hspace{1cm}\text{on  }[0,\min\{\tau,T\})
\end{equation}
in the case $\tb=0$, and
%$$\wb(t)\leq\frac{1}{t-\tb}\hspace{0.5cm}\Rightarrow$$
\begin{equation*}
W\leq\max\left\{\frac{1}{t-\tb},\frac{\phi(C)}{c}\right\}\hspace{1cm}\text{on  }[\tb,\min\{\tb+\tau,T\})
\end{equation*}
for a general $\tb$. Then we also have
\begin{equation}\label{w2}
W\leq\frac{\phi(C)}{c}\hspace{0.5cm}\text{on  }\left[\tb+\frac{c}{\phi(C)},\min\{\tb+\tau,T\} \right).
\end{equation}
Since $\tb$ is arbitrary, combining \eqref{w1} and \eqref{w2} and
using the second condition of \eqref{condh}, we obtain
$$W\leq\max\left\{\max_{\mo}W,\frac{\phi(C)}{c} \right\}\hspace{1cm}\text{on  }t \in[0,T),$$
which implies the assertion, since $\phi \leq d W$ by \eqref{du}. 
\end{proof}
\begin{cor}\label{Hub} $H$ and h are uniformly bounded on $[0,T)$.
\end{cor}
\begin{proof}
The boundedness of $H$ follows from Proposition \ref{p2} and property $i)$ of $\phi$, while the boundedness of $h$ follows from the boundedness of $\phi$.
\end{proof}
%By Corollary \ref{Hub} and Proposition \ref{conv} 
Since all principal curvatures are uniformly bounded on $[0,T)$,  $\mt$ can be written locally as a graph with uniformly bounded $C^{2,1}$ norm on balls of fixed radius. More precisely, the following result holds, see e.g. Lemma 3.4 in \cite{Sch2} and Section 8 in \cite{AM}.

\begin{prop}\label{gr}There exist $r,\eta>0$ depending only on $\sup H$ with the following property. Given any $(\bar{x},\tb)\in\m\times[0,T)$, there is a neighbourhood $\mathcal{U}$ of the point $\pb:=F(\bar{x},\tb)$ such that $\mt\cap\mathcal{U}$ coincides with the graph of a smooth function
$$u:B_r\times J\longrightarrow\mathbb{R}\hspace{1.5cm}\forall t\in J$$
where $B_r=B(\bar{p},r)\cap T_{\bar{p}}\m_{\tb}$ and $J:=(\max\{\tb-\eta,0\},\min\{\tb+\eta),T\}$. In addition, the $C^{2,1}$ norm of $u$ is uniformly bounded by a constant depending only on $\sup H$.
\end{prop}
We can use this representation to deduce the global existence of the solution. We recall some formulas valid for a hypersurface locally parametrized as the graph of a function $u$. Here $D_i$ and $D^i$ denote the derivatives with respect to local Euclidean coordinates.
\begin{align*}
g_{ij}&=\delta_{ij}+D_{i}uD_{j}u & g^{ij}&=\delta^{ij}-\frac{D^{i}uD^{j}u}{1+|Du|^2}\\
h_{ij}&=\frac{D^2_{ij}u}{(1+|Du|^2)^{1/2}} & H&=\frac{1}{(1+|Du|^2)^{1/2}}\left(\delta^{ij}-\frac{D^{i}uD^{j}u}{1+|Du|^2}\right)D^2_{ij}u
\end{align*}
Then, the flow \eqref{fl} is equivalent to
\begin{equation}\label{gfl}\partial_t u=(1+|Du|^2)^{1/2}\left\{\phi\left(\frac{1}{(1+|Du|^2)^{1/2}}\left(\delta^{ij}-\frac{D^{i}uD^{j}u}{1+|Du|^2}\right)D^2_{ij}u\right)-h\right\}.
\end{equation}
%Now we can deduce the global existence of the solution.

\begin{theorem}\label{ge} The solution $\mt$ of the flow \eqref{fl} exists for any time.
\end{theorem}
\begin{proof}
The preservation of the strict convexity implies that the flow is parabolic for all times. We also know that all curvatures stay bounded uniformly on $[0,T)$. To get  bounds on all derivatives of curvatures on any finite time interval, consider the equation \eqref{gfl}.
Denote
\begin{align*}
&F(D^2 u, Du, u, x, t):=
\\&\hspace{2cm}(1+|Du|^2)^{1/2}\left\{\phi\left(\frac{1}{(1+|Du|^2)^{1/2}}\left(\delta^{ij}-\frac{D^{i}uD^{j}u}{1+|Du|^2}\right)D^2_{ij}u\right)-h\right\}\\
%\end{align*}
%\begin{align*}
&\dot{F}^{ij}:=\frac{\partial F}{\partial D^2_{ij}u} \hspace{2cm} \ddot{F}^{ij,kl}:=\frac{\partial^2 F}{\partial D^2_{ij}u \, \partial D^2_{kl}u}.
\end{align*}
Notice that, on $B_r\times J$:
\begin{enumerate}
\item there exist constants $\lambda, \Lambda>0$ such that
$$\lambda \, {\rm Id}\leq\dot{F}^{ij}\leq\Lambda \, {\rm Id}.$$
In fact, % for any $k=1,\dots n$
$$\dot{F}^{ij}=\phi'\cdot\left(\delta^{ij}-\frac{D^iuD^ju}{1+|Du|^2}\right).$$
Given $w=\{w^i\}\in\mathbb{R}^n$,
$$\dot{F}^{ij}w^iw^j=\phi'\cdot\left(|w|^2-\frac{\langle Du,w\rangle^2}{1+|Du|^2}\right).$$
Then
%using Cauchy-Schwarz inequality we get
$$0<c\inf_{B_r\times J}\phi'|w|^2
\leq \dot{F}^{ij}w^iw^j\leq\sup_{B_r\times J}\phi'|w|^2<\infty$$
with $c>0$ a constant depending on $\sup_{B_r\times J}|Du|$.
%\vspace{1cm}
\item  Given any matrix $M_{ij}$ for which $\dot{F}^{ij}M_{ij}=0$, holds
$$\ddot{F}^{ij,kl}M_{ij}M_{kl}=0.$$
This trivially follows computing $\ddot{F}^{ij,kl}$.
\end{enumerate}
So we can apply Theorem $6$ in \cite{An2} to obtain a $C^{2,\alpha}$ estimate on $u$, for a suitable $\alpha \in(0,1)$. By standard parabolic theory, we can deduce uniform bounds on all higher derivatives of $u$. Covering $\mt$ with graphs over balls of radius $r$ we obtain  H\"older estimates for the curvature and its derivatives on any finite time interval, which imply that the maximal time of existence of the solution is infinite.
\end{proof}

\section{Convergence to a sphere}\label{Convergence}
\subsection*{Lower bound for the mean curvature}
In order to prove the convergence of the solution to a sphere, we need to use an argument similar to the one of Theorem \ref{ge} to obtain derivative bounds in the whole interval $t \in [0,+\infty)$. To do this, it is essential to have a positive lower bound on $H$, since Proposition \ref{conv} implies uniform convexity only on finite time intervals. Let us first give a preliminary result.
\begin{lem}\label{lpe}Given $\bar{t}\in[0,\infty)$, let $\bar{q}\in\Omega_{\bar{t}}$ be such that $\Omega_{\bar{t}}\subset B(\bar{q}, R^+)$, where $R^+$ is taken as in Corollary \ref{c1}. Then
$$\Omega_t\subset B(\bar{q},2R^+)\hspace{2cm}\forall t\in[\bar{t},\tb+\sigma]$$
where $\sigma>0$ is a constant that only depends on $n, |\mo|, |\Omega_0|$ and $\sup_{t} h(t)$.
\end{lem}
\begin{proof}
Let us compare $\mt$ with the sphere centered in $\qb$ whose radius $R(t)$ increases linearly according to
$$R(t)=\tilde{h}(t-\tb)+R^+$$
where $\tilde{h}=\sup_{t} h(t)$. $R(t)$ grows to $2R^+$ at time $t=\tb+\frac{R^+}{\tilde{h}}$. Denote $\sigma:=\frac{R^+}{\tilde{h}}$.

Similarly as in Lemma \ref{lpi}, set
$$r(x,t):=|F(x,t)-\qb|,  \hspace{1.5cm} u(x,t):=(F(x,t)-\qb,\nu(x,t)).$$
Then, the function $f(x,t):=R(t)-r(x,t)$ satisfies
$$\partial_t f=\tilde{h}-\frac{hu}{r}+\frac{\phi u}{r}\geq\frac{\phi u}{r}\geq0.$$
So $f(x,t)\geq0$ for every time, and $r(x,t) \leq R(\tb+\sigma)=2R^+$ for $t \in[\bar{t},\tb+\sigma]$.
\end{proof}

\begin{lem}\label{h} There exists $b>0$ such that 
$$h(t)\geq b\hspace{5mm}\forall t\in[0,\infty).$$  
\end{lem}

\begin{proof}
%We can unify together the two cases of $h(t)$ given by \eqref{apr} and \eqref{vpr}. In fact, from the proof of Lemma \ref{l1} in the area preserving case, we have
%$$\int_{\mt}\phi(H)H\geq\frac{1}{|\mt|}\int_{\mt}H\int_{\mt}\phi(H)$$
%Then
%$$h(t)=\frac{\int_{\mt} H\phi(H)}{\int_{\mt} H}\geq\frac{1}{|\mt|}\int_{\mt}\phi(H)$$
%So, even in the area preserving case, we can reduce to find a uniform lower bound on $\frac{1}{|\mt|}\int_{\mt}\phi(H)$. 
Let us first prove a bound from below on $\frac{1}{|\mt|}\int_{\mt}\phi(H) \, d\mu$. 
%From Corollary \ref{al}, we have
%\begin{equation}\label{hb}
%\frac{1}{|\mt|}\int_{\mt}\phi(H)\geq\frac{1}{M_2}\int_{\mt}\phi(H)
%\end{equation}
A well-known consequence of the Alexandrov-Fenchel inequalities is that there exists a constant $C_n$ only depending by $n$, such that 
$$\int_{\mt}H \, d\mu \geq C_n|\Omega_t|^{\frac{n-1}{n+1}}$$
so, by Corollary \ref{al} we get
$$\int_{\mt}H\, d\mu\geq C_0,$$
where $C_0>0$ is a constant depending by $n$ and the initial datum. 
By Corollary \ref{Hub}, there exists some value $H^*$ such that $H\leq H^*$ on $\mt$, for all $t$. Let $k=\frac{C_0}{2|M_0|}$ and $\tilde{\mt}=\{ x\in\m| H(x,t)\geq k\}$. Then we have
\begin{align*}
C_0&\leq\int_{\mt}H \, d\mu =\int_{\tilde{\mt}}H \, d\mu+\int_{\mt\smallsetminus\tilde{\mt}}H \, d\mu \leq H^*|\tilde{\mt}|+k|\mt|\\
&\leq H^*|\tilde{\mt}|+\frac{C_0}{2}
\end{align*}
where for the last inequality we used the fact that $|\mt|\leq|\mo|$ for all $t$. Thus
\begin{equation}
\label{mtb}
 |\tilde{\mt}|\geq\frac{C_0}{2H^*}.
 \end{equation}
Let us set  $m=\min_{k\leq H\leq H^*}\frac{\phi(H)}{H}$.
Using Corollary \ref{al}, we conclude
\begin{align*}
\frac{1}{|\mt|}\int_{\mt}\phi(H) \, d\mu &\geq\frac{1}{M_2}\int_{\mt}\phi(H) \, d\mu \geq\frac{1}{M_2}\int_{\tilde{\mt}}\phi(H) \, d\mu \\
&\geq\frac{m}{M_2}\int_{\tilde{\mt}}H \, d\mu \geq \frac{m}{M_2}|\tilde{\mt}|k\geq\frac{mC_0k}{2H^*M_2}>0,
\end{align*}
which gives a uniform bound from below on $h(t)$ in the volume-preserving case. In the area preserving case, the above computations also imply an estimate on $h(t)$ using the inequality
$$\int_{\mt}\phi(H)H \, d\mu\geq\frac{1}{|\mt|}\int_{\mt}H \, d\mu\int_{\mt}\phi(H) \, d\mu,$$
which was proved in the second part of the proof of Lemma \ref{l1}.

\end{proof}

To obtain a lower bound on $H$, we now use a technique analogous to Proposition \ref{p2}, but we reverse the sign of the test function by considering a ball which encloses $\mt$ instead of an enclosed one.  A similar argument was used in \cite{Schn} for an expanding flow. In contrast to the upper bound in Proposition \ref{p2}, the proof of the next result depends crucially on the presence of the nonlocal term $h(t)$.

\begin{prop}\label{p3}The mean curvature $H$ is uniformly bounded from below by a positive constant.
\end{prop}
\begin{proof}
Given any $\tb \geq 0$, let $\qb$ be chosen so that the conclusion of Lemma \ref{lpe} holds. We define
$$W(x,t):=\frac{\phi(H)}{c-u(x,t)}\hspace{1cm}c:=4R^+,$$
which is well defined on $[\tb,\tb+\sigma]$, because on such interval we have
$$\frac{c}{2}\leq c-u\leq c$$
where $u(x,t):=(F(x,t)-\qb,\nu(x,t))$ as usual.
Standard computations show that
\begin{align*}
\partial_t W&=
\phi'\Delta W +\frac{2\phi'}{c-u}(\nabla W,\nabla (c-u)) \\
& \hspace{0.5cm}-\frac{\phi'}{c-u}h|A|^2-\frac{\phi}{(c-u)^2}\{-h+\phi'H+\phi-c|A|^2\phi'\}.
\end{align*}
Now, define
$$\wbi(t):=\inf_{\mt}W(x,t)\hspace{1.5cm}Y(t):=\{x\in\m | W(x,t)=\wbi(t)\}.$$
Then, after disregarding the last positive term, we obtain
\begin{eqnarray}\label{ew2}
D_-\wbi & \geq & \inf_{Y(t)}\left\{-\frac{\phi'hH^2}{c-u}+\frac{h}{c-u}W-\frac{\phi'H}{c-u}W-W^2\right\} \nonumber \\
& \geq & \wbi \inf_{Y(t)}\left\{-\frac{\phi'hH^2}{\phi}+\frac{h}{c}-\frac{2\phi'H}{c}-\frac{2\phi}{c}\right\}.
\end{eqnarray}
%Using properties $iii)$ and $iv)$ of $\phi$, we can suppose that if $H$ is sufficiently small, then
%\begin{equation}\label{epw} \phi'H^2<\epsilon_1\phi\hspace{5mm}\text{and}\hspace{5mm}\phi'H<\epsilon_2
%\end{equation}
%where $\epsilon_1,\epsilon_2>0$ are approaching to zero. Because of the fact that $\phi(0)=0$, we can also assume that $W$ reaches its minimum when $H$ is very small, and so from  \eqref{ew2}, \eqref{epw} and Lemma \ref{h}, we obtain
%\begin{equation}\label{dwb}
%D_-\wbi\geq\wbi\inf_{Y(t)}\left\{-\epsilon_1\frac{h}{c-u}+\frac{m}{c-u}-\frac{\epsilon_2}{c-u}-W\right\}
%\end{equation}
%The quantity on the right side of \eqref{dwb} is positive, since we can suppose $H$ sufficiently small, say $H<\beta$, with $\beta>0$ a suitable constant. So, on any time interval $[\tb,\tb+\sigma]$, if $H<\beta$, then the minimum of $W$ is increasing in time. Then,
%$$W(x,t)\geq\min\left\{\wbi(0),\frac{\phi(\beta)}{c}\right\}\hspace{1cm}\text{on }[0,\infty)$$
%%such that $-W+\frac{h}{c-u}>\beta$, where $\beta$ is a suitable positive constant. So, by the maximum principle,
%and so $H$ is bounded from below  on the same time interval by a positive constant.

Using properties $i)$,  $iii)$ and $iv)$ of $\phi$, we can fix $\beta>0$ such that, if $H \in (0,\beta)$, we have
\begin{equation}\label{epw}
\frac{\phi'H^2}{\phi} < \frac{1}{2c}, \qquad
 \phi + \phi'H<\frac{b}{8}
\end{equation}
where $b>0$ is the lower bound on $h(t)$ given by Lemma \ref{h}. Suppose now that $\wbi(t) < \phi(\beta)/c$ at some time $t$. Then
$\phi(H) \leq \beta$ on $Y(t)$ and therefore
\begin{equation}\label{dwb}
D_-\wbi\geq\wbi \left\{-\frac{h}{2c}+\frac{h}{c}-\frac{b}{4c} \right\} \geq \frac b{4c} \wbi>0.
\end{equation}
This shows that $W$ cannot attain a new minimum smaller than $\phi(\beta)/c$, thus
$$W(x,t)\geq\min\left\{\wbi(0),\frac{\phi(\beta)}{c}\right\}\hspace{1cm}\text{on }[0,\infty).$$
%such that $-W+\frac{h}{c-u}>\beta$, where $\beta$ is a suitable positive constant. So, by the maximum principle,
From this we deduce that $\phi$, and so $H$, is bounded from below for all times by a positive constant.
\end{proof}

\subsection*{Smooth convergence to a sphere}
Proposition \ref{p3}, together with Corollary \ref{Hub}, implies that $H$ takes values in a fixed compact subset of $(0,+\infty)$ for all times. Therefore $\phi'(H)$ is bounded from above and below by positive constants for all $t \in [0,+\infty)$ and the flow is uniformly parabolic. Arguing as in the proof of Theorem \ref{ge}, we obtain that all derivatives of the curvatures are bounded for $t \in[0,\infty)$. So, by compactness, the hypersurfaces $\mt$ converge, up to time subsequences, to a smooth limit $\mathcal{M}_{\infty}$. To prove that this limit has to be a sphere, we  %Thus  the hypersurfaces converge to some smooth limit $\mathcal{M}_{\infty}$ as $t\to\infty$. We have to prove that this limit is a sphere. \\ 
% We will use this result to obtain  H\"older estimates on curvatures.
show that $\phi$ tends to its mean value.
\begin{prop}\label{hc} The velocity $\phi(H)$ tends uniformly to its mean value, i.e.
$$\lim_{t\to\infty}\max_{\mt}|\phi(H(x,t))-h(t)|=0$$
\end{prop}
\begin{proof}
We consider only the volume preserving case, since the area preserving case can be treated similarly.
For any $t$, let $\bar{H}(t)$ such that $\phi(\bar{H}(t))=h(t)$. Then we compute
\begin{align*}
\frac{d}{dt}|\mt|&=\int_{\mt}Hh \, d\mu-\int_{\mt}H\phi(H) \, d\mu\\
&=\int_{\mt}(H-\bar{H})(\phi(\bar{H})-\phi(H)) \, d\mu\\
&=-\int_{\mt}|H-\bar{H}||\phi(H)-\phi(\bar{H})| \, d\mu.
\end{align*}
Now, using the bound on $\phi'$ we obtain
\begin{align*}
\frac{d}{dt}|\mt|&\leq - \frac{1}{\sup\phi'} \int_{\mt}|\phi(H)-\phi(\bar{H})|^2 \, d\mu.\\
&=- \frac{1}{\sup\phi'}\int_{\mt}|\phi(H)-h|^2 \, d\mu.
\end{align*}
Suppose that $|\phi(H)-h|=a$ for some $a>0$ at some point $(\xb,\tb)$. The derivative bounds on the curvature imply that $H$ is uniformly Lipschitz continuous, and then there exists a radius $r(a)$, not depending by $(\xb,\tb)$, such that 
$$|\phi(H)-h|>\frac{a}{2}\hspace{1cm}\text{on }B((\xb,\tb),r(a))$$
where $B((\xb,\tb),r(a))$ is the parabolic neighbourhood centered at $(\xb,\tb)$ of radius $r(a)$. 
Then
\begin{equation}\label{as}
\frac{d}{dt}|\mt|<-\eta(a)\hspace{1cm}\forall t\in[\tb-r(a),\tb+r(a)]
\end{equation}
for some $\eta>0$ only depending on $a$.

By Lemma \ref{l1}, $|\mt|$ is positive and decreasing in time, and so property \eqref{as} can occur only on a finite number of time intervals, for any given $a>0$. This shows that $|\phi(H)-h|$ tends to zero uniformly.
\end{proof}
Proposition \ref{hc} implies that any possible limit of subsequences of $\mt$ has constant mean curvature, and so is a sphere. Standard techniques, see e.g. \cite{An1}, allow now to conclude that the whole family $\mt$ converges smoothly to a sphere. Thus the proof of Theorem \ref{mt} is complete. 
\bigskip\\
{\bf Acknowledgments}
Carlo Sinestrari was partially supported by the research group GNAMPA of INdAM (Istituto Nazionale di Alta Matematica).

\bigskip
\noindent Maria Chiara Bertini, Dipartimento di Matematica e Fisica, Universit\`a di Roma ``Roma Tre'', Largo San Leonardo Murialdo 1, 00146, Roma, Italy. E-mail: \\ bertini@mat.uniroma3.it \\

\noindent Carlo Sinestrari, Dipartimento di Ingegneria Civile e Ingegneria Informatica, Universit\`a di Roma ``Tor Vergata'', Via Politecnico 1, 00133, Roma, Italy. 
E-mail: sinestra@mat.uniroma2.it

\end{document}